\documentclass[11pt,a4paper,final]{amsart}

\usepackage{amssymb,latexsym}
\usepackage{enumerate}
\usepackage{pifont,wasysym}
\usepackage{ifsym} 
\usepackage[T1]{fontenc}
\usepackage{color}

\usepackage{shuffle}
\usepackage[left=3cm,right=3cm,top=3cm,bottom=3cm]{geometry} 
\usepackage{hyperref} 

\theoremstyle{plain}
\newtheorem{theorem}{Theorem}
\newtheorem{corollary}{Corollary}[section]
\newtheorem{lemma}[corollary]{Lemma}
\newtheorem{proposition}[corollary]{Proposition}

\newtheorem{theoremprim}{Theorem}

\theoremstyle{definition}

\theoremstyle{remark}
\newtheorem{remark}[corollary]{Remark}

\newcommand{\NN}{\mathbb{N}}
\newcommand{\ZZ}{\mathbb{Z}}
\newcommand{\RR}{\mathbb{R}}
\newcommand{\CC}{\mathbb{C} }

\newcommand{\Set}[1]{\left\{ #1  \right\}}
\newcommand{\SetSuchThat}[2]{\left\{\, #1 \ | \ #2 \, \right\}}
\newcommand{\Span}{\operatorname{span}}

\newcommand{\Mod}{\, \mathrm{mod}\, }

\newcommand{\ScalarProduct}[3]{(#1| #2)_{#3}}

\newcommand{\TangentFields}[1]{\Gamma(#1)}

\newcommand{\LieAlgebra}[1]{\mathfrak{#1}}

\newcommand{\VectorFlow}[1]{\exp\left(#1\right)}

\newcommand{\Manifold}{\mathrm{M}}

\newcommand{\ShuffleWeightedProduct}{\cshuffle}
\newcommand{\ShuffleProduct}{\shuffle}

\newcommand{\EmptyWord}{\mathrm{1}}
\newcommand{\ExpShuffle}[1]{\exp_{\ShuffleProduct}(#1)}
\newcommand{\Antipode}[1]{{#1}^*}
\newcommand{\FlipHom}{F}

\newcommand{\Lie}[1]{\mathrm{Lie}(#1)}
\newcommand{\Ideal}[1]{\mathrm{I}(#1)}
\newcommand{\Polynomials}[2]{#1\langle#2\rangle}
\newcommand{\Series}[2]{#1\langle\langle#2\rangle\rangle}

\newcommand{\SeriesShuffle}[2]{{#1}_{\ShuffleProduct}\langle\langle#2\rangle\rangle}
\newcommand{\WordsInSeriesOfLength}[2]{S_{#2}^{#1}}
\newcommand{\AOne}{\LieAlgebra{A}_1}
\newcommand{\aOne}{\LieAlgebra{a}_1}
\newcommand{\FreeToTangent}{X}
\newcommand{\FreeMagma}[1]{\mathrm{M}(#1)}
\newcommand{\GoodExpressions}[1]{\mathrm{E}(#1)}

\newcommand{\MainSeries}{S}
\newcommand{\AOneSeries}{S_{\LieAlgebra{a}_1}}

\newcommand{\DecreasingProduct}{\overleftarrow\prod}

\newcommand{\OddSetC}[1]{\Delta^{#1}}
\newcommand{\OddHallSetC}[1]{\HallSetOfLength{odd}{#1}}

\newcommand{\ChronAlgIterInt}[2]{\mathrm{II}_{[0,#2]}(#1)}
\newcommand{\ChronProduct}{*}
\newcommand{\IterIntegral}[2]{\Phi_{#1}(#2)}

\newcommand{\Letters}{\mathrm{A}}
\newcommand{\HallWordsOfLength}[2]{\mathrm{W}_{#2}^{#1}}
\newcommand{\WordsOfLength}[1]{\Letters^*_{#1}}

\newcommand{\HallSetOfLength}[2]{\mathrm{H}_{#2}^{#1}}

\newcommand{\HallSetAOne}[1]{\mathrm{H}^{\LieAlgebra{a}_1}_{#1}}

\newcommand{\LettersGeneral}{\mathrm{'\! A}}
\newcommand{\HallWordsOfLengthGeneral}[2]{\mathrm{' W}_{#2}^{#1}}
\newcommand{\WordsOfLengthGeneral}[1]{\LettersGeneral^*_{#1}}
\newcommand{\HallBasisOfLengthGeneral}[2]{\mathrm{' B}_{#2}^{#1}}
\newcommand{\HallSetOfLengthGeneral}[2]{\mathrm{' H}_{#2}^{#1}}

\newcommand{\HallPol}[1]{P_{#1}}
\newcommand{\HallDualPol}[1]{\xi_{#1}}

\newcommand{\FunctionGamma}[1]{\alpha_{#1}}
\newcommand{\FunctionXi}[1]{s_{#1}}
\newcommand{\Foliage}[1]{f(#1)}

\begin{document}

\title[Solution of $\LieAlgebra{a}_1$-type Lie-Scheffers system and Riccati equation]{Explicit solutions of the $\LieAlgebra{a}_1$-type Lie-Scheffers system and a general  Riccati equation}

\author[G. Pietrzkowski]{Gabriel Pietrzkowski 
}

\address{Institute of Mathematics, Polish Academy of Sciences; \' Sniadeckich 8,
00-956 Warszawa, Poland; \phone +48 22 5544524.
}
\email{G.Pietrzkowski@impan.pl}

\begin{abstract}
For a general differential system $\dot x(t) = \sum_{d=1}^3 u_d(t)X_d$, where $X_d$ generates the simple Lie algebra of type $\LieAlgebra{a}_1$, we compute the explicit solution in terms of iterated integrals of products of $u_d$'s. As a byproduct we obtain the  solution of a general Riccati equation by infinite quadratures.
\end{abstract}



\maketitle

\section{Introduction}
\label{sec:Introduction}

Let $\Manifold$ be a finite dimensional manifold. Let $X_a, X_b, X_c \in \TangentFields{\Manifold}$ be smooth tangent vector fields on $\Manifold$, generating the $\LieAlgebra{a}_1$-type Lie subalgebra, of the Lie algebra $\TangentFields{\Manifold}$, such that
\begin{align}
\label{eq:a1LieBrackets}
[X_a,X_b] &= 2X_a, & [X_a, X_c] &= -X_b, & [X_b, X_c] = 2 X_c,
\end{align}
where $[\cdot,\cdot]:\TangentFields{\Manifold}\otimes\TangentFields{\Manifold} \to\TangentFields{\Manifold}$ is a standard Lie bracket of tangent vector fields. Now for arbitrary measurable functions $u_a, u_b, u_c :[0,T]\to\RR$, consider a differential equation on $\Manifold$ defined as follows:
\begin{align}
\label{eq:DifferentialEquation}
\begin{split}
\dot x(t) &= u_c(t) X_c + u_b(t) X_b + u_a(t) X_a, \\
x(0) &= x_0 \in \Manifold.
\end{split}
\end{align}
The goal is to write a general (local) solution for this system in terms of flows of $X_a, X_b, X_c \in \TangentFields{\Manifold}$. This solution will  depend explicitly on the functions $u_a, u_b, u_c$ and their iterated integrals only. 

The motivation to consider such a problem comes from many directions.
For example, when $u_a, u_b$ and $u_c$ are constrained so as to define a control system, the solution described above gives rise to an explicit formula for the endpoint mapping of the control system \cite{Agrachev04Control}. 

Another example is when $\Manifold = \mathrm{SL}(2)$ is the special linear group, and $X_d$ are left invariant vector fields. In this case the connections with Riccati equation are well known (see \cite{Redheffer56Solutions,Redheffer57Riccati} and \cite{Carinena07Riccati,Carinena11Riccati}), as well as the subsequent applications to the study of Milne-Pinney equations (see \cite{Carinena09MilnePinney}), Ermakov systems (see \cite{Carinena08Ermakov}), and harmonic oscillators (see \cite{Carinena10HarmonicOscillator}). In particular our approach gives rise to the explicit solution of a general Riccati equation (see Theorem \ref{thm:Riccati}).

The advantage of our approach is that it can be generalized to systems of any simple-lie-algebra-type, in particular it will be very interesting to write out similar solutions for rank-two simple Lie algebras.

In order to solve the stated problem, we use the theorem of Kawski and Sussmann \cite{Kawski97NoncommutativePower}. The origins of their solution, come from the works of KT Chen \cite{Chen54Iterated,Chen57Integration,Chen68Algebraic} on iterated integration of paths, also called algebraic paths, and later application of Chen's results to control systems by Fliess \cite{Fliess81Fonctionnelles}. An important step was given by Sussmann \cite{Sussman86ProductExpansion}, and in strictly algebraic fashion by Melan{\c{c}}on and Reutenauer \cite{Melancon89LyndonWords}, where he expressed his solution of non-linear control-affine system, in terms of a Lyndon basis. This was then generalized in Reutenauer's book \cite{Reutenauer93FreeLie} (who actually claims that the original idea comes from Sch\"{u}tzenberger and Malan{\c{c}}on), and rewritten in a control-theoretical setting by Kawski and Sussmann \cite{Kawski97NoncommutativePower}.

The article is organized as follows. In section 2 we formulate the results; in section 3 we collect important preliminary facts from the theory of free Lie algebras and its connection with integration of differential systems; in section 4 we introduce a Hall set/basis adjusted to our problem, which allows us to prove our results in section 5. Finally, we conclude the paper in section 6.

\section{Results}
\label{sec:Results}

Before we state the main theorem, we need a few definitions. Let $\Letters = \Set{a, b, c}$ be a set of letters.  By $\WordsOfLength{n}$ denote the set of words composed from $n$ letters, i.e. 
$$
\WordsOfLength{n} = \SetSuchThat{a_1 a_2\cdots a_n}{\ a_1,\ldots,a_n\in\Letters},
$$ 
and by $\WordsOfLength{} = \bigcup_{n\in\NN\cup\Set{0}} \WordsOfLength{n}$ the set of all words. In particular $\WordsOfLength{0} = \Set{\EmptyWord}$ contains only one -- empty -- word. It is important to note that $\WordsOfLength{}$, with an associative, noncommutative product (called the concatenation product) given by juxtaposition $(v,w)\mapsto vw$, and the neutral element $\EmptyWord$, is a monoid.
We denote by $|w|$ the length of $w\in\WordsOfLength{}$, i.e. $|w| = n$ for $w\in \WordsOfLength{n}$, and by $|w|_d$ the number of occurrences of the letter $d$ in the word $w\in\WordsOfLength{}$.

The concatenation product gives rise to the $\RR$-algebras $\Polynomials{\RR}{\Letters}$ of noncommutative polynomials on $\Letters$, and $\Series{\RR}{\Letters}$ of non-commutative power series on $\Letters$.  
We denote by $\ShuffleProduct:\Series{\RR}{\Letters}\otimes \Series{\RR}{\Letters} \to \Series{\RR}{\Letters}$ the standard bilinear shuffle product  (see section \ref{sec:Preliminaries} for the definition), and denote by $\SeriesShuffle{\RR}{\Letters}$ the commutative $\RR$-algebra of power series on $\Letters$ with the product $\ShuffleProduct$ (we use this notation only in case the shuffle algebra structure is crucial).
We define the shuffle exponential $\exp_{\ShuffleProduct} : \Series{\RR}{\Letters}_0 \to \Series{\RR}{\Letters}$ by
$$
\ExpShuffle{P} := \EmptyWord + \sum_{k=1}^\infty \frac{P^{\ShuffleProduct k}}{k!},
$$
where the "0" subscript in $\Series{\RR}{\Letters}_0$ means that we take $P$ with zero constant term. 

For fixed measurable controls $u_a, u_b, u_c:[0,t]\to\RR$, we define the linear homomorphism $\Upsilon^t:\SeriesShuffle{\RR}{\Letters}\to \RR$ given by
\begin{align}
\label{eq:Upsilon}
\WordsOfLength{}\ni w = a_1\cdots a_n \mapsto \Upsilon^t(w) := \int_0^{t}u_{a_n}(t_n) \int_0^{t_n}\cdots \int_0^{t_2}u_{a_1}(t_1) \, dt_1\ldots dt_{n-1} dt_n.
\end{align}
It is easy to check that $\Upsilon^t$ is in fact an associative algebra homomorphism, i.e. $\Upsilon^t(v\ShuffleProduct w) = \Upsilon^t(v)\Upsilon^t(w)$ (see \cite{Chen68Algebraic,Reutenauer93FreeLie,Kawski02Combinatorics}).

Finally denote by $\VectorFlow{t X}:\Manifold\to\Manifold$ the flow of a tangent vector field $X\in\TangentFields{\Manifold}$ at time $t \geq 0$.

\begin{theorem}
\label{thm:Main}
Let $X_a, X_b, X_c \in \TangentFields{\Manifold}$ be smooth tangent vector fields satisfying \eqref{eq:a1LieBrackets}. Then (locally)
the solution $x:[0,T]\to\Manifold$ of the differential equation \eqref{eq:DifferentialEquation} is of the form
\begin{align}
\label{eq:Solution}
x(t) = \VectorFlow{\Xi_c(t) X_c} \VectorFlow{\Xi_b(t) X_b} \VectorFlow{\Xi_a(t) X_a} (x_0).
\end{align}
Here, $\Xi_a, \Xi_b, \Xi_c : [0,T] \to \RR$ are given by $\Xi_d(t) := \Upsilon^t(\WordsInSeriesOfLength{d}{})$ (for $d = a, b, c$), where 
\begin{align*}
 \WordsInSeriesOfLength{a}{} &= a\, \ExpShuffle{2\AOneSeries}, &
\WordsInSeriesOfLength{b}{} &= {\AOneSeries}, &
\WordsInSeriesOfLength{c}{} &= \ExpShuffle{2\AOneSeries}\, c, 
\end{align*}
and $\AOneSeries\in\Series{\RR}{\Letters}$  
is the unique solution of the algebraic equation
\begin{align*}
\AOneSeries =  b - a\, \ExpShuffle{2\AOneSeries}\, c.
\end{align*}
In particular, we have
\begin{align*}
b - a\WordsInSeriesOfLength{c}{} = \WordsInSeriesOfLength{b}{} = b - \WordsInSeriesOfLength{a}{} c.
\end{align*}
\end{theorem}

\begin{remark}
The solution in the form (\ref{eq:Solution}) exists locally by virtue of the Wei and Norman theorem \cite{WeiNorman1964OnGlobal}. It is not our goal to investigate the issue of the global solution in this paper (in the mentioned article Wei and Norman give example of a local but non-global solution in $\LieAlgebra{so}(3)$ case), rather, we focus on the formal one. 
\end{remark}

\begin{remark}
Since $\Upsilon^t:\SeriesShuffle{\RR}{\Letters}\to\RR$ is an algebra homomorphism, it is easy to check that
\begin{align*}
\Xi_c(t) = \int_0^t u_c(t_1)\cdot e^{2\Upsilon^{t_1}(\AOneSeries)}\, dt_1. 
\end{align*}

\end{remark}


As a byproduct of the proof of the above theorem we get that $\Xi_a$ is the solution of a Riccati equation. We state it in the following theorem.

\begin{theorem}
\label{thm:Riccati}
For fixed measurable functions $u_a, u_b, u_c:[0,T]\to\RR$ the function $\Xi_a:[0,T]\to\RR$, defined in Theorem \ref{thm:Main} by $\Xi_a(t) = \Upsilon^t(a\, \ExpShuffle{2\AOneSeries})$, is (locally) the solution of the Riccati equation:
\begin{align*}
\dot\Xi_a(t) &= u_a(t) + 2u_b(t)\cdot \Xi_a(t) - u_c(t)\cdot \Xi_a^2(t) \\
\Xi_a(0) &= 0.
\end{align*}
\end{theorem} 

As is well known, a specific solution of a Riccati equation gives rise to all its other solutions (the appropriate formula can be found for example in \cite{Carinena08Geoemetric}). An easy calculation gives the following corollary.

\begin{corollary}
For fixed measurable functions $u_a, u_b, u_c:[0,T]\to\RR$, the function $y : [0,T]\to\RR$, where
\begin{align*}
y(t) = \Upsilon^t (a\, \ExpShuffle{2\AOneSeries})  + \frac{y_0\cdot e^{2\Upsilon^t(\AOneSeries)}}{1 + y_0\cdot \int_0^t u_c(t_1)\cdot e^{2\Upsilon^{t_1}(\AOneSeries)}\, dt_1},
\end{align*}
is (locally) the solution of the Riccati equation:
\begin{align*}
\dot y(t) &= u_a(t) + 2u_b(t)\cdot y(t) - u_c(t)\cdot y^2(t) \\
y(0) &= y_0.
\end{align*}
\end{corollary}

The series $\AOneSeries\in\Series{\RR}{\Letters}$ can be computed explicitly. In order to do this, we recursively define two series of ordered sets of words $
\HallWordsOfLength{b}{n}, \HallWordsOfLength{c}{n} \subset \WordsOfLength{n}$ ($n=1,2,\ldots$) as follows (in each set we denote by $\leq$ the according total ordering):
\begin{enumerate}[(a)]
\item $\HallWordsOfLength{b}{1}:=\Set{b}$ and for $n \geq 2$, $\HallWordsOfLength{b}{n} = a\HallWordsOfLength{c}{n-1} := \SetSuchThat{av}{v\in\HallWordsOfLength{c}{n-1}}$ with ordering taken from $\HallWordsOfLength{c}{n-1}$ that is $av < aw$ iff $v < w$.
\item $\HallWordsOfLength{c}{1} := \Set{c}$ and for $n \geq 2$, 
\begin{enumerate}[(i)]
\item $\HallWordsOfLength{c}{n} := \bigcup_{i=1}^k \HallWordsOfLength{b}{i}\HallWordsOfLength{c}{n-i} := \bigcup_{i=1}^k \SetSuchThat{vw}{v\in\HallWordsOfLength{b}{i}, w\in\HallWordsOfLength{c}{n-i}}$ if $n=2k$ is even, with "the lexicographical ordering" that is $vw < v' w'$ with $v, v'\in\HallWordsOfLength{b}{i}$ and $w, w'\in\HallWordsOfLength{c}{n-i}$ iff $v < v'$ or ($v = v'$ and $w < w'$);
\item $\HallWordsOfLength{c}{n} := \bigcup_{i=1}^k \HallWordsOfLength{b}{i}\HallWordsOfLength{c}{n-i} \cup \OddSetC{k}$ if $n=2k+1$ is odd, where
$$
\OddSetC{k} := \SetSuchThat{avw}{v, w\in\HallWordsOfLength{c}{k},\ v \geq w }.
$$
The ordering is like in the case (i) with additional relations that $v < w $ for all $v\in\OddSetC{k}$, $w\in \bigcup_{i=1}^k \HallWordsOfLength{b}{i}\HallWordsOfLength{c}{n-i}$,  and $avw < av' w'$ with $v,v', w,w' \in\HallWordsOfLength{c}{k}$ iff $v < v'$ or ($v = v'$ and $w < w'$).
\end{enumerate}
\end{enumerate} 
For $d = 
b, c$, define the set $\HallWordsOfLength{d}{} := \bigcup_{n=1}^\infty \HallWordsOfLength{d}{n}$. In each of these sets we introduce the ordering originating from the orderings $\leq$ on $\HallWordsOfLength{d}{n}$ with the additional relation that for $v\in\HallWordsOfLength{d}{p}, w\in\HallWordsOfLength{d}{q}$ with $p > q$ we put $v\leq w$.

For all words in $\HallWordsOfLength{}{} :=  
\HallWordsOfLength{b}{} \cup \HallWordsOfLength{c}{} \subset \WordsOfLength{}$, we recursively define a mapping $\FunctionXi \cdot:\HallWordsOfLength{}{}\to\Polynomials{\RR}{\Letters}$ by 
\begin{enumerate}[(a)]
\item $\FunctionXi b := b$ and \\ $\FunctionXi {av} := a\FunctionXi v$ for $av\in \HallWordsOfLength{a}{n}$, where  $v\in\HallWordsOfLength{c}{n-1}$.
\item $\FunctionXi c := c$, \\ $\FunctionXi {bw} := b\FunctionXi v$ for $bw\in \HallWordsOfLength{c}{n}$, where  $w\in\HallWordsOfLength{c}{n-1}$, and \\ $\FunctionXi {avw} = a\FunctionXi v\ShuffleWeightedProduct\FunctionXi w$ for $avw\in \HallWordsOfLength{c}{n}$, where  $v,w\in\HallWordsOfLength{c}{n-1}$.
\end{enumerate}
The product $\ShuffleWeightedProduct$ is a slight modification of $\ShuffleProduct$, in particular, for $v_1 = \cdots = v_{i_1} > \cdots > w_1 = \cdots = w_{i_k}$, where $i_1,\ldots,i_k\in\NN$, we define
\begin{align*}
\FunctionXi {v_1}\ShuffleWeightedProduct\cdots\ShuffleWeightedProduct\FunctionXi {w_{i_k}} := 
\frac 1 {i_1!\cdots i_k !}\ \FunctionXi {v_1}^{\ShuffleProduct i_1} \ShuffleProduct\cdots\ShuffleProduct\FunctionXi {w_1}^{\ShuffleProduct i_k},
\end{align*}
where ${P}^{\ShuffleProduct 1} := P$ and ${P}^{\ShuffleProduct i} :=  {P}^{\ShuffleProduct i-1}\ShuffleProduct P$, $i=2,3,\ldots$, for all $P\in\Series{\RR}{\Letters}$.

Note that in case (b), the definition is enough, since for $vw\in\HallWordsOfLength{c}{}$, where $v\in\HallWordsOfLength{b}{} and w\in\HallWordsOfLength{c}{} $, we have $v=av'$ with $v'\in\HallWordsOfLength{b}{}$. 

Finally, we define the main series $\MainSeries\in\Series{\RR}{\Letters}$ by
\begin{align}
\label{eq:MainSeries}
\MainSeries := \sum_{w\in\HallWordsOfLength{b}{}} \FunctionGamma {b} (w)\, \FunctionXi{w},
\end{align}
where $\FunctionGamma {b}(b) := 1$ and 
\begin{align}
\label{eq:FunctionGammaB}
\FunctionGamma {b}(aw') := - (-2)^{|w'|_a}\cdot 2^{|w'|_b}
\end{align}
for $aw'\in\HallWordsOfLength{b}{}$.

\begin{proposition}
\label{prop:MainSeries}
The series $\MainSeries\in\Series{\RR}{\Letters}$ defined in \eqref{eq:MainSeries}, is the unique solution of the algebraic equation
\begin{align*}
\MainSeries =  b - a\, \ExpShuffle{2\MainSeries}\, c.
\end{align*}
In other words $\AOneSeries = \MainSeries$.
\end{proposition}

The expansion of the main series up to words of length less than $6$, is 
\begin{align*}
\MainSeries = b - ac - 2abc - 4abbc + 2aacc - 8abbbc + 4a(acb + 2 abc + bac)c + \ldots
\end{align*}
This expansion suggests a certain symmetry of the series $\MainSeries$. Namely, define an algebra antihomomorphism $\Antipode{(\cdot)}:\Series{\RR}{\Letters}\to\Series{\RR}{\Letters}$, $\Antipode{(a_1\cdots a_n)} = a_n\cdots a_1$ for $a_1\cdots a_n \in \WordsOfLength{}{}$, and an algebra homomorphism $\FlipHom:\Series{\RR}{\Letters}\to\Series{\RR}{\Letters}$, such that $\FlipHom(a) = c$, $\FlipHom(b) = b$, and $\FlipHom(c) = a$. Denote by $\Antipode\FlipHom:\Series{\RR}{\Letters}\to\Series{\RR}{\Letters}$ the composition of $\Antipode{(\cdot)}$ with $\FlipHom$.

\begin{proposition}
\label{prop:MainSeriesSymmetry}
Let $\WordsInSeriesOfLength{a}{}, \WordsInSeriesOfLength{b}{}, \WordsInSeriesOfLength{c}{}\in \Series{\RR}{\Letters}$ be the series defined in Theorem \ref{thm:Main}. Then 
\begin{align*}
\Antipode\FlipHom (\WordsInSeriesOfLength{b}{}) &= \WordsInSeriesOfLength{b}{},& 
\Antipode\FlipHom (\WordsInSeriesOfLength{a}{}) &= \WordsInSeriesOfLength{c}{},&
\Antipode\FlipHom (\WordsInSeriesOfLength{c}{}) &= \WordsInSeriesOfLength{a}{}.
\end{align*}
\end{proposition}

Finally, we compute an expression for $\ExpShuffle{2\MainSeries}$.

\begin{proposition}
\label{prop:ExpShuffle}
For the main series $\MainSeries\in\Series{\RR}{\Letters}$, it follows that
\begin{align*}
\ExpShuffle{2\MainSeries} = 
1 + \sum_{i=1}^\infty 2^i \sum_{\substack{w_1\geq\cdots\geq w_i \\ w_i\in\HallWordsOfLength{b}{}}} \FunctionGamma b(w_1) \cdots \FunctionGamma b(w_i)\ \FunctionXi {w_1}\ShuffleWeightedProduct\cdots\ShuffleWeightedProduct \FunctionXi {w_i}.
\end{align*}
\end{proposition}

\section{Preliminary results}
\label{sec:Preliminaries}

The proofs of Theorems \ref{thm:Main} and \ref{thm:Riccati}, are based on the results of Kawski and Sussmann \cite[section 5]{Kawski97NoncommutativePower} which generalize Sussmann's work \cite{Sussman86ProductExpansion}. Actually, a similar solution stated in a strictly algebraic fashion, can be found in  \cite[Corollary 5.6]{Reutenauer93FreeLie} (the algebraic analog of the earlier Sussmann's result was proved in \cite{Melancon89LyndonWords}).
In order to state these results, we introduce necessary notions from the theory of free Lie algebras (we follow \cite{Reutenauer93FreeLie} where the proofs and details can be found).

Assume $\LettersGeneral$ is a certain set which we call the alphabet. As in section \ref{sec:Results}, we denote by $\WordsOfLengthGeneral{n}$, $\WordsOfLengthGeneral{}$, $\Polynomials{\RR}{\LettersGeneral}$, and $\Series{\RR}{\Letters}$, the set of words of length $n$, the set of all words,  the $\RR$-algebras of non-commutative polynomials, and series in the letters $\LettersGeneral$, respectively. 
Since $\Polynomials{\RR}{\LettersGeneral}\subset\Series{\RR}{\LettersGeneral}$, and $\Series{\RR}{\LettersGeneral}$ is the algebraic closure of $\Polynomials{\RR}{\LettersGeneral}$, we define all objects in the larger algebra.
In particular, the product of two series $P=\sum_{w\in\WordsOfLength{}}\ScalarProduct{P}{w}{} w$ and $Q=\sum_{w\in\WordsOfLength{}}\ScalarProduct{Q}{w}{} w$ is defined by $PQ:=\sum_{w\in\WordsOfLength{}}\ScalarProduct{PQ}{w}{} w$ with $\ScalarProduct{PQ}{w}{} := \sum_{uv=w} \ScalarProduct{P}{u}{}\ScalarProduct{Q}{v}{}$.
In $\Series{\RR}{\LettersGeneral}$ we define the standard Lie bracket $[\cdot,\cdot]:\Series{\RR}{\LettersGeneral} \otimes\Series{\RR}{\LettersGeneral} \to\Series{\RR}{\LettersGeneral}$ given by $[P,Q] := PQ - QP$ for $P,Q\in\Series{\RR}{\LettersGeneral}$. 
We denote by $\Lie{\LettersGeneral}$ the smallest $\RR$-submodule of $\Series{\RR}{\LettersGeneral}$ which contains $\LettersGeneral$, and is closed under the Lie bracket. As is well known, $\Lie{\LettersGeneral}$ is the algebraic closure of the free Lie algebra generated by $\LettersGeneral$.

On $\Series{\RR}{\LettersGeneral}$ we also consider the bilinear shuffle product $\ShuffleProduct:\Series{\RR}{\LettersGeneral} \otimes\Series{\RR}{\LettersGeneral}\to \Series{\RR}{\LettersGeneral}$ defined recursively for words by putting $\EmptyWord\ShuffleProduct w = w\ShuffleProduct \EmptyWord = w$ for any $w\in\WordsOfLengthGeneral{}$, and 
\begin{align}
\label{eq:ShuffleDef}
(w_1a_1)\ShuffleProduct(w_2a_2) = (w_1\ShuffleProduct(w_2a_2))a_1 +  ((w_1a_1)\ShuffleProduct w_2)a_2
\end{align}
for all $a_1,a_2\in\LettersGeneral$ and $w_1, w_2\in\WordsOfLengthGeneral{}$. 
It easy to check, that \eqref{eq:ShuffleDef} is equivalent to
\begin{align}
\label{eq:ShuffleDefPrim}
(a_1w_1)\ShuffleProduct(a_2w_2) = a_1(w_1\ShuffleProduct(a_2w_2)) +  a_2((a_1w_1)\ShuffleProduct w_2)
\end{align}
for all $a_1,a_2\in\LettersGeneral$ and $w_1, w_2\in\WordsOfLengthGeneral{}$.

Let $\FreeMagma{\LettersGeneral}$ be the set of binary, complete, planar, rooted trees with leaves labelled by $\LettersGeneral$. Each such tree can be naturally identified with the unique expression in the set $\GoodExpressions{\LettersGeneral}$ defined by the following two conditions: (i) if $a\in\LettersGeneral$, then $a\in\GoodExpressions{\LettersGeneral}$, and (ii) if $t,t'\in\GoodExpressions{\LettersGeneral}$, then $(t,t')\in\GoodExpressions{\LettersGeneral}$. In the sequel we will not distinguish between these sets, i.e. we assume $\FreeMagma{\LettersGeneral} = \GoodExpressions{\LettersGeneral}$. Define the mapping
$\Foliage{\cdot}:\FreeMagma{\LettersGeneral}\to\WordsOfLengthGeneral{}$, which assigns to a tree $t\in\FreeMagma{\LettersGeneral}$ the word given by dropping all brackets in it, i.e., $\Foliage{a} = a$ for all $a\in\LettersGeneral$, and $\Foliage{(t,t')} = \Foliage{t}\Foliage{t'}$ for all $t,t'\in\FreeMagma{\LettersGeneral}$. The word $\Foliage{t}$ is called the foliage of $t\in\FreeMagma{\LettersGeneral}$.
Define also the mapping $\HallPol{\cdot}:\FreeMagma{\LettersGeneral}\to \Lie{\LettersGeneral}$, which changes the rounded brackets into the Lie brackets, i.e.,
$\HallPol{a} = a$ for all $a\in\LettersGeneral$, and $\HallPol{(t,t')} := [\HallPol{t},\HallPol{t'}]$ for all $t,t'\in\FreeMagma{\LettersGeneral}$.
We will generalize this definition in the sequel. 
A Hall set $\HallSetOfLengthGeneral{}{}$ on the letters $\LettersGeneral$ (which should also be called a Shirshov set \cite{Shirshov62Bases} and a Viennot set \cite{Viennot78Algebres}), is a subset of $\FreeMagma{\LettersGeneral}$ totally ordered by $\leq$ and satisfying:
\begin{enumerate}[(I)]
\item \label{enum:H1} $\LettersGeneral \subset \HallSetOfLengthGeneral{}{}$;
\item \label{enum:H2} if $h = (h',h'')\in\HallSetOfLengthGeneral{}{}\setminus\LettersGeneral$, then $h'' \in \HallSetOfLengthGeneral{}{}$ and $h < h''$;
\item \label{enum:H3} for all $h = (h',h'')\in\FreeMagma{\LettersGeneral}\setminus\LettersGeneral$ we have $h\in\HallSetOfLengthGeneral{}{}$ iff 
\begin{itemize}
\item $h', h'' \in \HallSetOfLengthGeneral{}{}$ and $h' < h''$, and
\item $h' \in \LettersGeneral$ or $h' = (x,y)$ such that $y \geq h''$. 
\end{itemize}
\end{enumerate}

Fix a Hall set $\HallSetOfLengthGeneral{}{}$ on the letters $\LettersGeneral$ totally ordered by $\leq$. Each Hall tree $h\in\HallSetOfLengthGeneral{}{}$ corresponds to a word $\Foliage{h}\in\WordsOfLengthGeneral{}$ called a Hall word. Denote by $\HallWordsOfLengthGeneral{}{}$, the set of Hall words with ordering $\leq$ inherited from the ordering on $\HallSetOfLengthGeneral{}{}$ in the natural way.
It is a nontrivial fact that each word $w\in\WordsOfLengthGeneral{}$, is the unique concatenation of a unique non-increasing series of Hall words, that is, $w=h_1\cdots h_k$ for some unique $k\in\NN$, and $h_i\in\HallWordsOfLengthGeneral{}{}$ such that $h_1 \geq\cdots\geq h_k$  (in the sequel we will use letter '$h$' to describe both Hall words and Hall trees).
Let $\HallPol{\cdot} : \WordsOfLengthGeneral{}\to\Polynomials{\RR}{\LettersGeneral}$ be the mapping defined by
\begin{enumerate}[(i)]
\item $\HallPol{\EmptyWord} := 1$;
\item $\HallPol{a} := a$ for $a\in\LettersGeneral$;
\item $\HallPol{h} := \HallPol{t}\in\Lie{\LettersGeneral}$ for $h\in\HallWordsOfLengthGeneral{}{}$ such that $h=\Foliage{t}$, $t\in\HallSetOfLengthGeneral{}{}\subset\FreeMagma{\LettersGeneral}$;
\item $\HallPol{w} := \HallPol{h_1}\cdots\HallPol{h_k}\in\Polynomials{\RR}{\LettersGeneral}$ for $w = h_1\cdots h_k$, where $k\in\NN$ and $h_i\in\HallWordsOfLengthGeneral{}{}$ such that $h_1 \geq\cdots\geq h_k$.
\end{enumerate} 
The set $\HallBasisOfLengthGeneral{}{} := \SetSuchThat{\HallPol{h}\in\Lie{\LettersGeneral}} {h\in\HallWordsOfLengthGeneral{}{}}$ is the Hall basis of $\Lie{\LettersGeneral}$  corresponding to the Hall set $\HallSetOfLengthGeneral{}{}$. 
By the Poincar{\' e}-Birkhoff-Witt theorem, the set of ordered products $\HallPol{h_1}\cdots \HallPol{h_k}$, where $h_1 \geq \cdots \geq h_k$ are Hall words, creates a basis for the enveloping algebra of $\Lie{\LettersGeneral}$, which in the free case is isomorphic to $\Series{\RR}{\LettersGeneral}$. Therefore $\SetSuchThat{\HallPol{w}}{w\in\WordsOfLengthGeneral{}}$ is a basis in $\Series{\RR}{\LettersGeneral}$ (but in fact each $\HallPol{w}\in\Polynomials{\RR}{\LettersGeneral}$, so it is also a basis in $\Polynomials{\RR}{\LettersGeneral}$). For our purpose it is crucial to consider the dual basis $\SetSuchThat{\HallDualPol{w}}{w\in\WordsOfLengthGeneral{}}$ of the algebra $\Series{\RR}{\LettersGeneral}$, defined as always by 
\begin{align*}
v = \sum_{w\in\WordsOfLengthGeneral{}}\ \ScalarProduct{\HallDualPol{w}}{v}{} \HallPol{w}
\end{align*} 
for any $v\in\WordsOfLengthGeneral{}$. In this context we have the following proposition.

\begin{proposition}[{\cite[Theorem 5.3]{Reutenauer93FreeLie}}]
\label{prop:HallDualPol}
\begin{enumerate}[(i)]
\item $\HallDualPol{1} = 1$;
\item If $h = av \in \HallWordsOfLengthGeneral{}{}$ is a Hall word, where $a\in \LettersGeneral, v\in\WordsOfLengthGeneral{}$, then $\HallDualPol{h} = a\HallDualPol{v}$; 
\item If $w = h_1^{i_1}\cdots h_k^{i_k} \in \WordsOfLengthGeneral{}$ is any word, where $h_1 > \cdots > h_k$ are Hall words and $i_1,\ldots,i_k\in\NN$, then 
\begin{align*}
\HallDualPol{w} = 
\frac 1 {i_1!\cdots i_k !}\ \HallDualPol{h_1}^{\ShuffleProduct i_1} \ShuffleProduct\cdots\ShuffleProduct\HallDualPol{h_k}^{\ShuffleProduct i_k}
\end{align*}
\end{enumerate}
(recall that ${P}^{\ShuffleProduct 1} = P$ and ${P}^{\ShuffleProduct i} =  {P}^{\ShuffleProduct i-1}\ShuffleProduct P$, $i=2,3,\ldots$, for all $P\in\Polynomials{\RR}{\LettersGeneral}$).
\end{proposition}

In what follows we will sometimes use a natural notation $\HallDualPol{h} := \HallDualPol{\Foliage h}$ for $h\in\HallSetOfLength{}{}$.

In \cite{Kawski97NoncommutativePower}, Kawski and Sussmann consider the so called universal control system evolving in the algebra $\Series{\RR}{\LettersGeneral}$, and the solution is known as the Chen-Fliess series. Their main result states that  the Chen-Fliess series is equal to the infinite product of group-like elements in $\Series{\RR}{\LettersGeneral}$ (i.e., elements $\exp(S)$ where $S\in\Lie{\LettersGeneral}$), parametrized by any Hall set on the letters ${\LettersGeneral}$. For our purposes it is important that such a universal control system can be utilized in examining a differential affine control system:
\begin{align}
\label{eq:GeneralControlSystem}
\begin{split}
\dot y(t) &= \sum_{a\in \LettersGeneral} u_a(t) Y_a,\\
y(0) &= y_0\in \Manifold,
\end{split}
\end{align}
where $\Manifold$ is an arbitrary finite dimensional smooth manifold, $u_a:[0,T]\to\RR$ are measurable controls, and $Y_i\in\TangentFields{\Manifold}$ are fixed smooth tangent vector fields.
We state the Kawski-Sussmann theorem for this control system. Namely, if we take any Hall set $\HallSetOfLengthGeneral{}{}$, then for fixed $u_i$'s (locally) the solution to \eqref{eq:GeneralControlSystem} is
\begin{align}
\label{eq:GeneralSolution}
y(t) = \DecreasingProduct_{h\in\HallSetOfLengthGeneral{}{}} \VectorFlow{\IterIntegral{h}{t}\cdot Y_h}(y_0).
\end{align}
The symbol $\DecreasingProduct_{h\in\HallSetOfLengthGeneral{}{}}$ denotes the decreasing product with respect to the ordering $\leq$ in $\HallSetOfLengthGeneral{}{}$, i.e. 
$$ 
\DecreasingProduct_{h\in\HallSetOfLengthGeneral{}{}}F_h := \cdots F_h F_{h'} \cdots,
$$ 
where $ \cdots > h > h' > \cdots$.
The tangent vector fields $Y_h$, are defined by the relation of the Hall set with the Hall basis in $\Lie{\LettersGeneral}$, and the universal property of the free Lie algebra $\Lie{\LettersGeneral}$, that is $Y_. : \HallSetOfLengthGeneral{}{}\to \TangentFields{\Manifold}$ is the composition of the mapping $\HallSetOfLengthGeneral{}{} \ni h \mapsto \HallPol{h}\in\Lie{\LettersGeneral}$ with the unique Lie-algebra homomorphism generated by $\Lie{\LettersGeneral} \supset \LettersGeneral \ni a \mapsto Y_a \in \TangentFields{\Manifold}$.

In order to define $\IterIntegral{h}{t}$, we introduce the Zinbiel algebra $\ChronAlgIterInt{\LettersGeneral}{t}$ of iterated integrals of controls $u_a(t), a\in\LettersGeneral$, that is the $\RR$-algebra generated by $U_a (t) := \int_0^t u_a(t_1)\, dt_1$ with Zinbiel product $\ChronProduct : \ChronAlgIterInt{\LettersGeneral}{t}\otimes \ChronAlgIterInt{\LettersGeneral}{t} \to \ChronAlgIterInt{\LettersGeneral}{t}$ given by 
$$
(f_1\ChronProduct f_2)(t) :=  \int_0^t f_1(t_1) f_2'(t_1)\, dt_1.
$$  
The algebra is Zinbiel since $f_1\ChronProduct (f_2\ChronProduct f_3) = (f_1\ChronProduct f_2)\ChronProduct f_3 + (f_2\ChronProduct f_1)\ChronProduct f_3$ for all $f_1, f_2, f_3\in\ChronAlgIterInt{\LettersGeneral}{t}$, which is easy to check. In particular, 
$$
((\cdots(U_{a_1}\ChronProduct U_{a_2})\ChronProduct\cdots)\ChronProduct U_{a_n})(t) = 
\int_0^{t}u_{a_n}(t_n) \int_0^{t_n}\cdots \int_0^{t_2}u_{a_1}(t_1) \, dt_1\ldots dt_{n-1} dt_n.
$$
Therefore, we define the linear mapping $\Upsilon^t_{\int}:\Series{\RR}{\LettersGeneral}\to \ChronAlgIterInt{\LettersGeneral}{t}$ by putting $\Upsilon^t_{\int}(\EmptyWord) = 1$ and
\begin{align}
\label{eq:UpsilonDef}
\Upsilon^t_{\int}(a_1\cdots a_n) &:= \Upsilon^t_{\int}(a_1\cdots a_{n-1})\ChronProduct U_{a_n}
\\
\nonumber
&=
((\cdots(U_{a_1}\ChronProduct U_{a_2})\ChronProduct\cdots)\ChronProduct U_{a_n})(t)
\end{align}
for any word $a_1\cdots a_n\in\WordsOfLengthGeneral{}$.
For fixed integrable controls $u_a$, we define $\Upsilon^t:\SeriesShuffle{\RR}{\LettersGeneral}\to\RR$, $S \mapsto \Upsilon^t(S)$ as the evaluation of $\Upsilon^t_{\int}(S) \in \ChronAlgIterInt{\LettersGeneral}{t}$ on these controls. This definition coincides with \eqref{eq:Upsilon}. Since $\HallDualPol{h}$ are expressed in terms of the shuffle product, it is important to note  that \cite{Chen68Algebraic,Reutenauer93FreeLie,Kawski02Combinatorics}
\begin{align}
\label{eq:UpsilonIsAHomom}
\Upsilon^t(v\ShuffleProduct w) = \Upsilon^t(v)\cdot\Upsilon^t(w),
\end{align}
where $\cdot$ is the ordinary multiplication in $\RR$.

Now the remaining definition in \eqref{eq:GeneralSolution} is 
\begin{align}
\label{eqn:IterIntegral}
\IterIntegral{h}{t} := \Upsilon^t (\HallDualPol{\Foliage{h}}).
\end{align}

\section{$\LieAlgebra{a}_1$-type Hall set}
\label{sec:HallSet}

In this section we begin to prove the theorems stated in section \ref{sec:Results}. As we pointed out in the previous section, we are going to use formula \eqref{eq:GeneralSolution} of the formal solution of \eqref{eq:GeneralControlSystem}. 
In order to do that, we construct a specific Hall set $\HallSetAOne{}$ on the set of letters $\LettersGeneral = \Letters = \Set{a,b,c}$ ordered by $\leq$, adapted to Lie algebra of the type $\LieAlgebra{a}_1$. Let $\Ideal{\LieAlgebra{a}_1}$ be the smallest Lie ideal of $\Lie{\Letters}$ generated by elements of the form $[a,b] - 2a,  [a, c] + b, [b, c] - 2 c$. Then $\AOne := \Lie{\Letters} / \Ideal{\LieAlgebra{a}_1}$ is a Lie algebra of the type $\LieAlgebra{a}_1$. The Hall set $\HallSetAOne{}$ is going to satisfy the following conditions:
\begin{enumerate}[(I)]
\item \label{enum:HS1} $\HallSetAOne{} = \HallSetOfLength{o}{}\cup\HallSetOfLength{a}{}\cup \HallSetOfLength{b}{}\cup\HallSetOfLength{c}{}$; if $h\in\HallSetOfLength{d}{}$ ($d=a,b,c$), then 
\begin{align}
\label{eq:HallPolGamma}
\HallPol{h} = \gamma_d(h)\cdot d \Mod \Ideal{\aOne}
\end{align}
with certain functions $\gamma_d:\HallSetOfLength{d}{}\to\ZZ$, and if $h\in\HallSetOfLength{o}{}$, then $\HallPol{h} = 0 \Mod \Ideal{\aOne}$;
\item \label{enum:HS2} $\HallSetOfLength{d}{} = \bigcup_{n\in\NN} \HallSetOfLength{d}{n}$ for $d=o,a,b,c$, and $h\in\HallSetOfLength{d}{n}$ iff $h\in\HallSetOfLength{d}{}$ and $|h| = n$;
\item \label{enum:HS3} each set $\HallSetOfLength{d}{n}$ is finite.
\end{enumerate}
Assuming we define $\HallSetOfLength{d}{n}$ with total orderings $\leq$ of the type (cardinality of  $\HallSetOfLength{d}{n}$) $\in\NN$, we define the total ordering $\leq$ in $\HallSetAOne{}$ of the type $\omega\cdot 4$ (i.e., the ordering isomorphic to the lexicographical ordering on $\Set{c,b,a,o}\times\NN$) by adding the following intuitive relations:
\begin{enumerate}[(A)]
\item \label{enum:Order1} for $h_1\in\HallSetOfLength{d}{n_1}, h_2\in\HallSetOfLength{d}{n_2}$ with $n_1\neq n_2$ we have $h_1 < h_2$ iff $n_1 > n_2$;
\item \label{enum:Order2} for $h_o\in\HallSetOfLength{o}{}, h_a\in\HallSetOfLength{a}{}, h_b\in\HallSetOfLength{b}{}, h_c\in\HallSetOfLength{c}{}$ we have $h_o < h_a < h_b < h_c$.
\end{enumerate}

Let us comment on these assumptions. The reason for distinguishing such a Hall set is that although we consider the free Lie algebra, to write out the solution of the differential equation, we actually assume certain relations on the vector fields involved. That is why we define the ideal $\Ideal{\aOne}$, and the quotient algebra $\AOne$. Namely, if we extend the mapping $\Letters \to \TangentFields{\Manifold}$, $d\mapsto X_d$ for $d= a,b,c$, to the Lie algebra homomorphism $ X_\cdot: \Lie{\Letters}\to\TangentFields{\Manifold}$, then $\Ideal{\aOne}$ is in the kernel of $X_\cdot$, so we can also consider the quotient algebra homomorphism $\FreeToTangent_\cdot:\AOne\to\TangentFields{\Manifold}$. The point is that, by \eqref{enum:HS1}, $\FreeToTangent_h = \gamma_d(h) X_d$ for each $h\in\HallSetOfLength{d}{}$ if $d=a,b,c$, and $\FreeToTangent_h = 0$ for each $h\in\HallSetOfLength{o}{}$. In particular, we do not need to care about elements or the ordering of $\HallSetOfLength{o}{}$.

Let us construct the sets $\HallSetOfLength{d}{n}$.
By the definition, each Hall set contains all letters, so if \eqref{enum:HS1} and \eqref{enum:HS2} (for n=1) are to be satisfied, we define $\HallSetOfLength{d}{1} := \Set{d}$ for $d =a, b, c$ and $\HallSetOfLength{o}{1} := \emptyset$. By \eqref{enum:Order2} we have $a < b < c$. In order to satisfy the definition of a Hall set for $n=2$, we must allocate  $(a,b), (a,c), (b,c)$. To fulfill \eqref{enum:HS1} we put 
\begin{align*}
\HallSetOfLength{a}{2} &:= \Set{(a,b)},& \HallSetOfLength{b}{2} &:= \Set{(a,c)},& \HallSetOfLength{c}{2} &:= \Set{(b,c)}, & \HallSetOfLength{o}{2} &:= \emptyset.
\end{align*} 
By \eqref{enum:Order2}, the ordering is $(a,b) < a < (a,c) < b < (b,c) < c$.
We pass to $n=3$. We need to add $((a,b),a), ((a,b),b), (a,(a,c)), (a,(b,c)), ((a,c),b), ((a,c),c), (b,(b,c)), ((b,c),c)$. Now it is easy to compute Lie polynomials connected with these elements modulo $\Ideal{\aOne}$, and to satisfy \eqref{enum:HS1} we must put
\begin{align*}
\HallSetOfLength{a}{3} &:= \Set{(a,(a,c)),((a,b),b)},& \HallSetOfLength{b}{3} &:= \Set{(a(bc))},\\ \HallSetOfLength{c}{3} &:= \Set{(b,(b,c)),((a,c),c)}, & \HallSetOfLength{o}{3} &:= \Set{((a,b),a), ((a,c),b), ((b,c),c)}.
\end{align*} 
We order these sets so that $(a,(a,c)) > ((a,b),b)$, $(b,(b,c)) > ((a,c),c)$ and $\HallSetOfLength{o}{3}$ in any arbitrary way. 
For $n>3$ the number of Hall elements grow very fast, but observe that in fact we are interested only in specific Hall elements. Namely, in $\HallSetOfLength{a}{n}$ there will be only (but not all) elements of the form $(v,w)$, where $v\in\HallSetOfLength{a}{k}, w\in\HallSetOfLength{b}{l}$, $k,l < n$. This comes from the relations $[a,b] = 2a, [a,c] = - b, [b,c] = 2c$ satisfied in $\AOne$ together with the assumptions \eqref{enum:H2} (in the definition of a Hall set) and \eqref{enum:Order1} (the assumption on the ordering). Similarly, if $(v,w)\in\HallSetOfLength{b}{n}$, then $v\in\HallSetOfLength{a}{k}, w\in\HallSetOfLength{c}{l}$, $k,l < n$, and if $(v,w)\in\HallSetOfLength{c}{n}$, then $v\in\HallSetOfLength{b}{k}, w\in\HallSetOfLength{c}{l}$, $k, l < n$. Adding the assumption \eqref{enum:H3} of a Hall set, we reach the following lemma.

\begin{lemma}
\label{lem:HallSet}
Assume that $\HallSetOfLength{a}{1} := \Set{a}, \HallSetOfLength{b}{1} := \Set{b}, \HallSetOfLength{c}{1} := \Set{c}, \HallSetOfLength{o}{1} := \emptyset$. For $n > 1$ we recursively define ordered sets $\HallSetOfLength{d}{n}\subset\FreeMagma{\Letters}$ by
\begin{enumerate}[(a)]
\item $\HallSetOfLength{a}{n} := \SetSuchThat{(\cdots(a,v_1)\cdots ,v_i)\in\FreeMagma{\Letters}} {v_i\in\HallSetOfLength{b}{},\ v_1 \geq\cdots\geq v_i,\ |v_1|+\ldots+ |v_i| = n-1}$ with any total ordering.
\item $\HallSetOfLength{b}{n} := \SetSuchThat{(a,v)\in\FreeMagma{\Letters}}{v\in\HallSetOfLength{c}{n-1}}$ with ordering taken from $\HallSetOfLength{c}{n-1}$ that is $(a,v) < (a,w)$ iff $v < w$;
\item 
\begin{enumerate}[(i)]
\item in case $n = 2k$ is even, 
$$
\HallSetOfLength{c}{n} := \bigcup_{i=1}^k \SetSuchThat{(v,w)\in\FreeMagma{\Letters}} {v\in\HallSetOfLength{b}{i}, w\in \HallSetOfLength{c}{n-i}},
$$ 
with "the lexicographical ordering" that is $(v,w) < (v', w')$ with $v, v'\in\HallSetOfLength{b}{i}$ and $w, w'\in\HallSetOfLength{c}{n-i}$ iff $v < v'$ or ($v = v'$ and $w < w'$);
\item in case $n = 2k+1$ is odd, 
\begin{align*}
\HallSetOfLength{c}{n} := \bigcup_{i=1}^k \SetSuchThat{(v,w)\in\FreeMagma{\Letters}} {v\in\HallSetOfLength{b}{i}, w\in \HallSetOfLength{c}{n-i}} \cup \OddHallSetC{k},
\end{align*} 
where $\OddHallSetC{k} := \SetSuchThat{((a,v),w)\in\FreeMagma\Letters} {v,w \in\HallSetOfLength{c}{k},\ v \geq w }.$ The ordering is like in case (i) with additional relations that $v < w $ for all $v\in\OddHallSetC{k}$, $w\in \HallSetOfLength{c}{n}\setminus\OddHallSetC{k}$,  and $((a,v),w) < ((a,v'), w')$ with $v,v', w,w' \in\HallSetOfLength{c}{k}$ iff $v < v'$ or ($v = v'$ and $w < w'$).
\end{enumerate}
\item $\HallSetOfLength{o}{n} := \HallSetOfLength{aa}{n}\cup \HallSetOfLength{oa}{n}\cup \HallSetOfLength{bb}{n}\cup \HallSetOfLength{ob}{n}\cup \HallSetOfLength{cc}{n} \cup \HallSetOfLength{oc}{n}\cup \HallSetOfLength{oo}{n}$ with any fixed ordering, where
\begin{align*}
\HallSetOfLength{d d'}{n} := \bigcup_{i=1}^{n-1}\SetSuchThat{(h',h'')\in \FreeMagma{\Letters}}{h'\in\HallSetOfLength{d}{i}, h''\in\HallSetOfLength{d'}{n-i}, \text{s.t. (\ref{enum:H2}) and (\ref{enum:H3}) are satisfied} }.
\end{align*}
\end{enumerate} 
Then $\HallSetAOne{}$ defined by \eqref{enum:HS1}, \eqref{enum:HS2} and \eqref{enum:HS3} with the ordering $\leq$ given by \eqref{enum:Order1} and \eqref{enum:Order2} is a Hall set.
\end{lemma}

\begin{remark}
The elements in $\HallSetOfLength{o}{}$ are not important for our purposes, and in fact we could consider any other elements which span the space $\Span\SetSuchThat{\HallPol{h}}{h\in\HallSetOfLength{o}{}}\subset\Lie{\Letters}$ together with the elements in $\HallSetOfLength{a}{}\cup \HallSetOfLength{b}{} \cup\HallSetOfLength{c}{}$ give a Hall set.
\end{remark}

From (b) and (c) of the above lemma, we conclude that
\begin{corollary}
\label{cor:HallSetC}
In the set $\HallSetOfLength{c}{n}$ there are only elements of three types, i.e.,
\begin{enumerate}[(a)]
\item $((a,v_1),v_2)$, where $v_1\in\HallSetOfLength{c}{i-1}$, $v_2\in\HallSetOfLength{c}{n-i}$, where $1\leq i \leq (n+1)/2$, and  $v_1 \geq v_2$;
\item $(b,v)$, where $v\in\HallSetOfLength{c}{n-1}$;
\item $c\in\HallSetOfLength{c}{1}$.
\end{enumerate}
\end{corollary}

The above lemma is almost self explanatory, nevertheless we prove it.

\begin{proof}[Proof of Lemma \ref{lem:HallSet}.]
The condition (\ref{enum:H1}) in the definition of a  Hall set is obvious. 
Let us prove that the condition (\ref{enum:H2}) is satisfied. We proceed by induction on the length $n$ of the foliage of trees. The base assumption obviously comes from the condition (\ref{enum:H1}). Assume $n >1$.
If $h\in\HallSetOfLength{a}{n}$, then  it is of the form $h=(h',v)$, where $h'\in\HallSetOfLength{a}{n - |v|}$ and $v\in\HallSetOfLength{b}{|v|}$, so by \eqref{enum:Order2} we have $h < v$.  
If $h\in\HallSetOfLength{b}{n}$, then $h = (a,v)$, where $v\in\in\HallSetOfLength{c}{n - 1}$, so \eqref{enum:Order2} implies $h < v$.
If $h\in\HallSetOfLength{c}{n}$, then $h = (h',w)$, where $h'\in\HallSetOfLength{b}{n - |w|}$ and $w\in\HallSetOfLength{c}{|w|}$, so \eqref{enum:Order1} implies $h < v$. For $h\in\HallSetOfLength{o}{n}$ there is nothing to prove.

Now let us check the condition (\ref{enum:H3}). Assume $h = (h',h'')\in\HallSetAOne{}$. 
Then $h\in\HallSetOfLength{a}{n}$ iff $h'\in\HallSetOfLength{a}{n - |h''|}$ and $h''\in\HallSetOfLength{b}{|h''|}$, so  by \eqref{enum:Order2} $h' < h''$, and $h' = a$ or $h' = (x,y)$ with $x\in\HallSetOfLength{a}{|h'|-|y|}$ and $y\in\HallSetOfLength{b}{|y|}$ such that $y\geq h''$.
Also $h\in\HallSetOfLength{b}{n}$ iff $h'=a$ and $h''\in\HallSetOfLength{c}{n-1}$, so  by \eqref{enum:Order2} $a < h''$.
Finally, $h\in\HallSetOfLength{c}{n}$ iff $h'\in\HallSetOfLength{b}{n - |h''|}$ and $h''\in\HallSetOfLength{c}{|h''|}$, so  by \eqref{enum:Order2} $h' < h''$, and  by the above corollary $h' = b$ or $h' = (x,y)$ with $x = a$, $y\in\HallSetOfLength{c}{|y|}$ and $y \geq h''.$
For $h\in\HallSetOfLength{o}{n}$ once again there is nothing to prove.
\end{proof}

\begin{remark}
The descriptions of the sets $\OddHallSetC{k}$ are not explicit. But observe that if we assume $\HallSetOfLength{c}{k} := \Set{v_1,\ldots,v_q}$ with the ordering $v_1 > \cdots > v_q$,
then it is easy to see that
\begin{align*}
\OddHallSetC{k} = \bigcup_{i=1}^q \SetSuchThat{((a,v_i),w)\in\FreeMagma\Letters}{w = v_{i}, v_{i+1}, \ldots, v_q}.
\end{align*}
\end{remark}

In Table \ref{tab:HallElements} we show all the Hall words of type $a,b$ and $c$ of length less then 6. We abbreviate notation by omitting commas and the most exterior brackets.
\begin{table}[ht]
\centering
\begin{tabular}{cccc}
$\HallSetOfLength{d}{n}$ & $d=a$ & $d=b$ & $d=c$ \\
$n=1$ & $a$ & $b$ & $c$  \\
\hline
$n=2$ & $(ab)$ & $(ac)$  & $(bc)$ \\
\hline
$n=3$ & $a(ac)$ & $a(bc)$ & $b(bc)$ \\
& $(ab)b$ &  & $(ac)c$ \\
\hline
$n=4$ & $a(a(bc))$ & $a(b(bc))$ & $b(b(bc))$ \\
& $(ab)(ac)$ & $a((ac)c)$ & $b((ac)c)$ \\
& $((ab)b)b$ &  & $(ac)(bc)$ \\
\hline
n=5 & $a(a(b(bc)))$ & $a(b(b(bc)))$ & $b(b(b(bc)))$ \\
& $a(a((ac)c))$ & $a(b((ac)c))$ & $b(b((ac)c))$ \\
& $(ab)(a(bc))$ & $a((ac)(bc))$ & $b((ac)(bc))$ \\
& $(a(ac))(ac)$ && $(ac)(b(bc))$ \\
& $((ab)b)(ac)$ && $(ac)((ac)c)$ \\
& $(((ab)b)b)b$ && $(a(bc))(bc)$ \\
\hline
&$\vdots$ & $\vdots$ & $\vdots$
\end{tabular}
\caption{Hall elements of type $a,b$ and $c$ of the length less then 6. The ordering is from up to down, and then from right to left.}
\label{tab:HallElements}
\end{table}

\section{Proofs of Theorems}

In the previous section we constructed the Hall set $\HallSetAOne{}$ suitable for our problem. Now we use it together with the Kawski-Sussmann solution \eqref{eq:GeneralSolution}. Namely, we assume that $\LettersGeneral = \Letters = \Set{a,b,c}$, that $Y_d = X_d$ for $d= a,b,c$, and that the $Y_d's$ satisfy \eqref{eq:a1LieBrackets}.  Using the Hall set $\HallSetAOne{}$, we know by \eqref{eq:GeneralSolution}, that 
\begin{align}
x(t) = \DecreasingProduct_{h\in\HallSetAOne{}{}} \VectorFlow{\IterIntegral{h}{t}\cdot X_h}(x_0)
\end{align}
is a solution of \eqref{eq:DifferentialEquation}. But $\HallSetAOne{}$ has special properties, which we now use. Namely, $\HallSetAOne{} = \HallSetOfLength{o}{}\cup\HallSetOfLength{a}{}\cup \HallSetOfLength{b}{}\cup\HallSetOfLength{c}{}$, each $\HallSetOfLength{d}{}$ has the ordering of the type $\omega$, and for $h_o\in\HallSetOfLength{o}{}, h_a\in\HallSetOfLength{a}{}, h_b\in\HallSetOfLength{b}{}, h_c\in\HallSetOfLength{c}{}$, we have $h_o < h_a < h_b < h_c$. Therefore, 
\begin{align*}
\DecreasingProduct_{h\in\HallSetAOne{}} \VectorFlow{\IterIntegral{h}{t}\cdot X_h}(x_0) & = 
\DecreasingProduct_{h\in\HallSetOfLength{c}{}} \VectorFlow{\IterIntegral{h}{t}\cdot X_h}
\DecreasingProduct_{h\in\HallSetOfLength{b}{}} \VectorFlow{\IterIntegral{h}{t}\cdot X_h} \\
&\qquad\circ
\DecreasingProduct_{h\in\HallSetOfLength{a}{}} \VectorFlow{\IterIntegral{h}{t}\cdot X_h}
\DecreasingProduct_{h\in\HallSetOfLength{o}{}} \VectorFlow{\IterIntegral{h}{t}\cdot X_h}(x_0).
\end{align*}
Moreover, if $h\in\HallSetOfLength{d}{}$ ($d=a,b,c$), then $X_{h} = \gamma_d(h)\cdot X_d$ with certain well defined functions $\gamma_d:\HallSetOfLength{d}{}\to\ZZ$ (we give explicit expressions for $\gamma_b$ and $\gamma_c$ in Lemma \ref{lem:FunctionGamma}), and if $h\in\HallSetOfLength{o}{}$, then $X_{h} \equiv 0$. Thus it follows that
\begin{multline*}
x(t) = 
\DecreasingProduct_{h\in\HallSetAOne{}} \VectorFlow{\IterIntegral{h}{t}\cdot X_h}(x_0)  = 
\VectorFlow{\sum_{h\in\HallSetOfLength{c}{}} \IterIntegral{h}{t}\gamma_c(h)\cdot X_c} \\
\circ
\VectorFlow{\sum_{h\in\HallSetOfLength{b}{}} \IterIntegral{h}{t}\gamma_b(h)\cdot X_b}
\VectorFlow{\sum_{h\in\HallSetOfLength{a}{}} \IterIntegral{h}{t}\gamma_a(h)\cdot X_a}(x_0).
\end{multline*}
Using the definition \eqref{eqn:IterIntegral} of $\IterIntegral{h}{t}$ we get
\begin{align*}
\sum_{h\in\HallSetOfLength{d}{}} \IterIntegral{h}{t}\gamma_d(h) =  \Upsilon^t \left(\sum_{h\in\HallSetOfLength{d}{}} \gamma_d(h)\  \HallDualPol{h}\right)
\end{align*}
(recall, that we use a natural abbreviation $\HallDualPol{h} := \HallDualPol{\Foliage h}$ for $h\in\HallSetAOne{}$).
In order to prove Theorem \ref{thm:Main} we must show that 
\begin{align*}
\sum_{h\in\HallSetOfLength{a}{}} \gamma_a(h)\  \HallDualPol{h} &= a\, \ExpShuffle{2\AOneSeries}, &
\sum_{h\in\HallSetOfLength{b}{}} \gamma_b(h)\  \HallDualPol{h} &= {\AOneSeries}, &
\sum_{h\in\HallSetOfLength{c}{}} \gamma_c(h)\  \HallDualPol{h} &= \ExpShuffle{2\AOneSeries}\, c, 
\end{align*}
where $\AOneSeries\in\Series{\RR}{\Letters}$ is the unique series satisfying $\AOneSeries =  b - a\, \ExpShuffle{2\AOneSeries}\, c$. To do this, we first
prove Theorem \ref{thm:MainPrim} (stated below) in which we simply change the series $\AOneSeries\in\Series{\RR}{\Letters}$ from Theorem \ref{thm:Main} into the main series $\MainSeries$ (defined in (\ref{eq:MainSeries})), and then show that $\MainSeries$ is in fact the unique solution of the algebraic equation $\MainSeries =  b - a\, \ExpShuffle{2\MainSeries}\, c$, i.e., prove Proposition \ref{prop:MainSeries}.

\begin{theoremprim}
\label{thm:MainPrim}
Let $X_a, X_b, X_c \in \TangentFields{\Manifold}$ be smooth tangent vector fields satisfying \eqref{eq:a1LieBrackets}. Then (locally)
the solution $x:[0,T]\to\Manifold$ of the differential equation \eqref{eq:DifferentialEquation} is of the form
\begin{align}
\tag{\ref{eq:Solution}}
x(t) = \VectorFlow{\Xi_c(t) X_c} \VectorFlow{\Xi_b(t) X_b} \VectorFlow{\Xi_a(t) X_a} (x_0).
\end{align}
Here, $\Xi_a, \Xi_b, \Xi_c : [0,T] \to \RR$ are given by $\Xi_d(t) := \Upsilon^t(\WordsInSeriesOfLength{d}{})$ (for $d = a, b, c$), where 
\begin{align*}
 \WordsInSeriesOfLength{a}{} &= a\, \ExpShuffle{2\MainSeries}, &
\WordsInSeriesOfLength{b}{} &= {\MainSeries}, &
\WordsInSeriesOfLength{c}{} &= \ExpShuffle{2\MainSeries}\, c, 
\end{align*}
and $\MainSeries\in\Series{\RR}{\Letters}$ is defined in (\ref{eq:MainSeries}).
Additionally, we have
\begin{align}
\label{eq:Additional}
b - a\WordsInSeriesOfLength{c}{} = \WordsInSeriesOfLength{b}{} = b - \WordsInSeriesOfLength{a}{} c.
\end{align}
\end{theoremprim}

\subsection{Proof of Theorem \ref{thm:MainPrim}b with the first equality in (\ref{eq:Additional})}

The following lemma demonstrates the connection between the mapping $\FunctionXi{\cdot}:\HallWordsOfLength{}{}\to\Polynomials{\RR}{\Letters}$, defined in section \ref{sec:Results}, and $\HallDualPol{\cdot}:\HallSetAOne{}\to\Polynomials{\RR}{\Letters}$, via the restriction of the correspondence between the set of words $\WordsOfLength{}$ and the dual basis in $\Series{\RR}{\Letters}$.

\begin{lemma}
\label{lem:HallSetAndWords}
Consider the Hall set $\HallSetAOne{}$ with its ordering $\leq$ defined in section \ref{sec:HallSet}, and the set of words $\HallWordsOfLength{}{} = \bigcup_{d=b,c}\bigcup_{n=1}^{\infty}\HallWordsOfLength{d}{n}\subset \WordsOfLength{}$ with ordering $\leq$  defined in section \ref{sec:Results}. Then for $d=b,c$, and $n\in\NN$, the Hall trees in $\HallSetOfLength{d}{n}$ correspond with the Hall words in $\HallWordsOfLength{d}{n}$, and the correspondence respects the orderings, i.e., for each $h\in\FreeMagma{\Letters}$, $h\in\HallSetOfLength{d}{n}$ iff $\Foliage{h}\in\HallWordsOfLength{d}{n}$, and $v < w$ iff $\Foliage{v} < \Foliage{w}$ for all $v,w\in\HallSetOfLength{b}{}\cup\HallSetOfLength{c}{}$.
Moreover, for each $h\in \HallWordsOfLength{b}{}\cup\HallWordsOfLength{c}{}$, 
$$
\HallDualPol{h} = \FunctionXi{h}.
$$
\end{lemma}

\begin{remark}
In other words, the above lemma states that the elements of $\HallWordsOfLength{d}{n}$ are the Hall words corresponding to the Hall trees in $\HallSetOfLength{d}{n}$, for $d=b,c$ and $n\in\NN$.
\end{remark}

\begin{proof}
The correspondence $h\in\HallSetOfLength{d}{n} \Longleftrightarrow\Foliage{h}\in\HallWordsOfLength{d}{n}$ (which respects the orderings) is as trivial as erasing the brackets for the Hall trees, which in fact the foliage mapping does. 
An inductive proof of the equality  $\HallDualPol{h} = \FunctionXi{h}$ comes easily from Proposition \ref{prop:HallDualPol}, and the proved correspondence between elements in the Hall set and the Hall words, together with the definitions of these sets.
\end{proof}

\begin{lemma}
\label{lem:FunctionGamma}
The functions $\gamma_b:\HallSetOfLength{b}{}\to\RR$ (defined in (\ref{eq:HallPolGamma})) and $\FunctionGamma{b}:\HallWordsOfLength{b}{}\to\RR$ (defined in (\ref{eq:FunctionGammaB})) satisfy $\gamma_b(h) = \FunctionGamma{b}(\Foliage h)$ for all $h\in\HallSetOfLength{b}{}$, i.e.,
$\gamma_b(b) = 1$ and $\gamma_b((a,v)) = -(-2)^{|\Foliage v|_a}\cdot 2^{|\Foliage{v}|_b}$ for each $(a,v)\in\HallSetOfLength{b}{}$ ($v\in\HallSetOfLength{c}{}$). Moreover, $\gamma_c(v) = -\gamma_b((a,v))$ for each $v\in\HallSetOfLength{c}{}$.
\end{lemma}

\begin{proof}
Recall that each $\gamma_d:\HallSetOfLength{d}{}\to\RR$ is defined by the formula $\HallPol{h} = \gamma_d(h)\cdot d \Mod \Ideal{\aOne}$. We first prove that 
$\gamma_c(v) = (-2)^{|\Foliage v|_a}\cdot 2^{|\Foliage{v}|_b}$ for each $v\in\HallSetOfLength{c}{}$. We proceed by induction on the length of the foliage of $v$. For $v=c$ we obviously have $\gamma_c(c) = 1$. Assume the statement is satisfied for $v\in\HallSetOfLength{c}{1}\cup\ldots\cup \HallSetOfLength{c}{n-1}$ and take $w\in\HallSetOfLength{c}{n}$. By Corollary \ref{cor:HallSetC} we must consider two cases.

\noindent Case I. $w = (a,v)$ where $v\in\HallSetOfLength{c}{n-1}$. Then
\begin{align*}
\HallPol{w} = [b,\HallPol{v}] = [b,\gamma_c(v)\cdot c] \Mod \Ideal{\aOne} =
2\gamma_c(v)\cdot c \Mod \Ideal{\aOne}.
\end{align*}
Therefore, using the inductive hypothesis, it follows that
\begin{align*}
\gamma_c(w) = 2\cdot (-2)^{|\Foliage v|_a}\cdot 2^{|\Foliage{v}|_b} = (-2)^{|\Foliage w|_a}\cdot 2^{|\Foliage{w}|_b}.
\end{align*}
\noindent Case II. $w = ((a,v_1),v_2)$ where $v_1\in\HallSetOfLength{c}{i-1}$ and $v_2\in\HallSetOfLength{c}{n-i}$. We compute
\begin{align*}
\HallPol{w} = [[a,\HallPol{v_1}],\HallPol{v_2}] = 
-2\gamma_c(v_1)\gamma_c(v_2)\cdot c \Mod \Ideal{\aOne}.
\end{align*}
By the inductive hypothesis we get
\begin{align*}
\gamma_c(w) = (-2)\cdot (-2)^{|\Foliage v_1|_a+|\Foliage v_2|_a}\cdot 2^{|\Foliage{v_1}|_b+|\Foliage{v_2}|_b} = (-2)^{|\Foliage w|_a}\cdot 2^{|\Foliage{w}|_b}.
\end{align*}
Thus we have proved the formula for $\gamma_c$.

Now let us focus on $d = b$. Obviously $\gamma_b(b) = 1$. From Lemma \ref{lem:HallSet}, we know that an element $w\in\HallSetOfLength{b}{n}$ (for $n > 1$) is of the form $w=(a,v)$, where $v\in\HallSetOfLength{c}{n-1}$, so 
\begin{align*}
\HallPol{w} = [a,\HallPol{v}] =  -\gamma_c(v)\cdot b \Mod \Ideal{\aOne}.
\end{align*}
Therefore $\gamma_b((a,v)) = \gamma_c(v)$, which together with the formula for $\gamma_c$, completes the proof.
\end{proof}

From these two lemmas we conclude that 
\begin{align*}
\sum_{h\in\HallSetOfLength{b}{}} \gamma_b(h)\  \HallDualPol{h} &= 
\sum_{w\in\HallWordsOfLength{b}{}} \FunctionGamma b(w)\  \HallDualPol{w} =
{\MainSeries},
\end{align*}
which proves that $\WordsInSeriesOfLength{b}{} = \MainSeries$.

Additionally, from Lemma \ref{lem:HallSet}(b), Proposition \ref{prop:HallDualPol}, and Lemma \ref{lem:FunctionGamma}, we get
\begin{align*}
\sum_{h\in\HallSetOfLength{b}{}} \gamma_b(h)\  \HallDualPol{h} & = 
\gamma_b(b)\ \HallDualPol{b} + \sum_{v\in\HallSetOfLength{b}{}} \gamma_b((a,v))\  \HallDualPol{av} \\
&=
b + \sum_{v\in\HallSetOfLength{c}{}} - \gamma_c(v)\  a\,\HallDualPol{v} = 
b - a \sum_{v\in\HallSetOfLength{c}{}}  \gamma_c(v)\  \HallDualPol{v}.
\end{align*} 
Together with Lemma \ref{lem:HallSetAndWords}, the previous equation proves the first additional equality $b - a\,\WordsInSeriesOfLength{c}{} = \WordsInSeriesOfLength{b}{}$ of formula (\ref{eq:Additional}) in Theorem \ref{thm:MainPrim}.

\subsection{Proof of Theorem \ref{thm:MainPrim}a  and Proposition \ref{prop:ExpShuffle}}

In this section we calculate the expression $\WordsInSeriesOfLength{a}{}$ in the assertion of Theorem \ref{thm:MainPrim}. 

From Lemma \ref{lem:HallSet}(a), 
\begin{align}
\nonumber
\HallSetOfLength{a}{} & = \Set{a}\cup \bigcup_{n=2}^\infty
\SetSuchThat{(\cdots(a,v_1)\cdots ,v_i)} {v_i\in\HallSetOfLength{b}{},\ v_1 \geq\cdots\geq v_i,\ |v_1|+\ldots+ |v_i| = n-1} \\
\label{eq:HallSetA}
& = 
\SetSuchThat{(\cdots(a,v_1)\cdots ,v_i)\in\FreeMagma{\Letters}} {i=0,1,2,\ldots,\ v_i\in\HallSetOfLength{b}{},\ v_1 \geq\cdots\geq v_i}.
\end{align}

\begin{lemma}
\label{lem:HallElemA}
For $h = (\cdots(a,v_1)\cdots ,v_i)\in\HallSetOfLength{a}{}$ such that $v_i\in\HallSetOfLength{b}{}$ and $v_1 \geq\cdots\geq v_i$, we have
\begin{enumerate}[(i)]
\item $\gamma_a(h) = 2^{i}\cdot\gamma_b(v_1)\cdots\gamma_b(v_i)$
\item $\HallDualPol{h} = a\ \HallDualPol{v_1}\ShuffleWeightedProduct \cdots\ShuffleWeightedProduct \HallDualPol{v_i}.$
\end{enumerate}
\end{lemma}

The product $\ShuffleWeightedProduct$ in $\Series{\RR}{\Letters}$ is defined in analogy with the product defined in section \ref{sec:Results}.
Namely, 
for $v_1 = \cdots = v_{i_1}> \cdots > w_1 = \cdots = w_{i_k}$, where $i_1,\ldots,i_k\in\NN$, we define
\begin{align*}
\HallDualPol{v_1}\ShuffleWeightedProduct\cdots\ShuffleWeightedProduct\HallDualPol {w_{i_k}} := 
\frac 1 {i_1!\cdots i_k !}\ \HallDualPol{v_1}^{\ShuffleProduct i_1} \ShuffleProduct\cdots\ShuffleProduct\HallDualPol{w_1}^{\ShuffleProduct i_k}.
\end{align*}

\begin{proof}
(i) We proceed by induction on $i\in\NN$. For $i=0$, $\gamma_a(a) = 1$. Assume the assertion is satisfied for $i \leq j\in\NN$, and take $h = (\cdots(a,v_1)\cdots ,v_{j+1})\in\HallSetOfLength{a}{}$. By the definitions of $\gamma_a$ and $\HallPol{\cdot}$, $\HallPol{h}=[\cdots[a,v_1]\cdots ,v_{j+1}] = \gamma_a(h)\cdot a \Mod \Ideal{\aOne}$ and $\HallPol{(a,v_1)} = [a,v_1] = \gamma_a((a,v_1))\cdot a \Mod \Ideal{\aOne}$. By the inductive hypothesis, the latter is equal to $2\gamma_b(v_1)\cdot a \Mod \Ideal{\aOne}$, and it follows that 
\begin{align*}
\gamma_a(h)\cdot a \Mod \Ideal{\aOne} & = 2\gamma_b(v_1)\cdot
[\cdots[a,v_2]\cdots ,v_{j+1}] \Mod \Ideal{\aOne} \\ 
& =
2\gamma_b(v_1)\cdot 2^{j}\cdot\gamma_b(v_2)\cdots\gamma_b(v_{j+1})\cdot a \Mod \Ideal{\aOne},
\end{align*}
where the inductive hypothesis is again used in the last line. Thus $\gamma_a(h) = 2^{j+1}\cdot\gamma_b(v_1)\cdots\gamma_b(v_{j+1})$ and we are done.

\noindent (ii)
This is an immediate consequence of Proposition \ref{prop:HallDualPol}.  
\end{proof}

Using \eqref{eq:HallSetA} together with the above lemma we get

\begin{align*}
\sum_{h\in\HallSetOfLength{a}{}} \gamma_a(h)\  \HallDualPol{h}  
= a\left( 1 + \sum_{i=1}^\infty 2^i \sum_{\substack{v_1\geq\cdots\geq v_i \\ v_i\in\HallSetOfLength{b}{}}} \gamma_b(v_1) \cdots \gamma_b(v_i)\ \HallDualPol{v_1}\ShuffleWeightedProduct\cdots\ShuffleWeightedProduct \HallDualPol{v_i} \right).
\end{align*}
Now, using Lemmas \ref{lem:HallSetAndWords} and \ref{lem:FunctionGamma} with the notation $w_i=\Foliage{v_i}$, we get

\begin{align*}
\sum_{h\in\HallSetOfLength{a}{}} \gamma_a(h)\  \HallDualPol{h}  
=
 a\left( 1 + \sum_{i=1}^\infty 2^i \sum_{\substack{w_1\geq\cdots\geq w_i \\ w_i\in\HallWordsOfLength{b}{}}} \FunctionGamma{b}(w_1) \cdots \FunctionGamma{b}(w_i)\ \FunctionXi {w_1}\ShuffleWeightedProduct\cdots\ShuffleWeightedProduct \FunctionXi {w_i} \right),
\end{align*}
where $\FunctionGamma b(b) = 1$ and $\FunctionGamma b(av) = -(-2)^{|v|_a}\cdot 2^{|v|_b}$ for $av\in\HallWordsOfLength{b}{}$ ($v\in\HallWordsOfLength{c}{}$).
To end the proof of Theorem \ref{thm:MainPrim}(a), it remains to prove Proposition \ref{prop:ExpShuffle} which is the content of the following calculation:
\begin{align*}
\ExpShuffle{2\MainSeries} & = \EmptyWord + \sum_{i=1}^\infty \frac {2^i} {i!} \left(\sum_{w\in\HallWordsOfLength{b}{}} \FunctionGamma {b} (w)\, \FunctionXi{w}\right)^{\ShuffleProduct i} 
\\
&= \EmptyWord + \sum_{i=1}^\infty \frac {2^i} {i!}
\left(\sum_{w_1\in\HallWordsOfLength{b}{}} \FunctionGamma {b} (w_1)\, \FunctionXi{w_1}\right)
\ShuffleProduct\cdots\ShuffleProduct
\left(\sum_{w_i\in\HallWordsOfLength{b}{}} \FunctionGamma {b} (w_i)\, \FunctionXi{w_i}\right) 
\\
&=  \EmptyWord + \sum_{i=1}^\infty {2^i}
\sum_{w_1,\ldots,w_i\in\HallWordsOfLength{b}{}}
\FunctionGamma {b} (w_1)\cdots\FunctionGamma {b} (w_i)\, 
\frac{1}{i!}\FunctionXi{w_1} \ShuffleProduct\cdots\ShuffleProduct \FunctionXi{w_i}
\\
&= 1 + \sum_{i=1}^\infty 2^i \sum_{\substack{w_1 > \cdots > w_j \\ w_i\in\HallWordsOfLength{b}{}}} 
\sum_{\substack{(k_1,\ldots,k_j)\in\NN^j \\ k_1 + \ldots + k_j = i}}
(\FunctionGamma{b}(w_1))^{k_1} \cdots (\FunctionGamma{b}(w_j))^{k_j}
\\ 
&\qquad\qquad\qquad\qquad\qquad\qquad \times\frac{1}{k_1!\cdots k_j!}
\FunctionXi {w_1}^{\ShuffleProduct k_1}\ShuffleProduct\cdots\ShuffleProduct \FunctionXi {w_j}^{\ShuffleProduct k_j},
\\
&= 1 + \sum_{i=1}^\infty 2^i \sum_{\substack{w_1\geq\cdots\geq w_i \\ w_i\in\HallWordsOfLength{b}{}}} \FunctionGamma{b}(w_1) \cdots \FunctionGamma{b}(w_i)\ \FunctionXi {w_1}\ShuffleWeightedProduct\cdots\ShuffleWeightedProduct \FunctionXi {w_i}.
\end{align*}

\begin{remark}
Note that the above calculation is in fact very similar to the one Kawski and Sussmann gave proving their theorem \cite{Kawski97NoncommutativePower}.
\end{remark}

\subsection{End of proof of Theorem \ref{thm:MainPrim}}
\label{sec:End}

In this section we calculate the expression $\WordsInSeriesOfLength{c}{}$ and the second additional equality in the assertion of Theorem \ref{thm:MainPrim}. As a byproduct we also get Theorem \ref{thm:Riccati}.

At this point we know that (locally) the solution of the system 
\begin{align}
\label{eq:DifferentialEquationPrim}
\tag{\ref{eq:DifferentialEquation}}
\begin{split}
\dot x(t) &= u_c(t) X_c + u_b(t) X_b + u_a(t) X_a, \\
x(0) &= x_0 \in \Manifold.
\end{split}
\end{align}
can be written in the form
\begin{align}
\label{eq:SolutionPrim}
x(t) = \VectorFlow{\tilde\Xi_c(t) X_c} \VectorFlow{\Xi_b(t) X_b} \VectorFlow{\Xi_a(t) X_a} (x_0),
\end{align}
where
\begin{align*}
\tilde\Xi_c(t) =  \Upsilon^t \left(\sum_{h\in\HallSetOfLength{c}{}} \gamma_c(h)\  \HallDualPol{h}\right),
\end{align*}
and $\Xi_a(t) = \Upsilon^t(a\, \ExpShuffle{2\MainSeries)}, \Xi_b(t) = \Upsilon^t(\MainSeries)$.
Differentiating \eqref{eq:SolutionPrim} with respect to $t$ gives the following equation:
\begin{align*}
\dot x &= (-\Xi_a^2 e^{-2\Xi_b}\cdot\dot{\tilde\Xi}_c - 2\Xi_a\cdot\dot{\Xi}_b + \dot{\Xi}_a) X_a \\
&\quad +
(\Xi_a e^{-2\Xi_b}\cdot\dot{\tilde\Xi}_c + \dot{\Xi}_b) X_b \\
&\quad +
(e^{-2\Xi_b}\cdot\dot{\tilde\Xi}_c) X_c.
\end{align*}
Comparing this equation with \eqref{eq:DifferentialEquationPrim}, we see that
\begin{align}
\label{eq:Controls}
\begin{split}
u_a &= -\Xi_a^2 e^{-2\Xi_b}\cdot\dot{\tilde\Xi}_c - 2\Xi_a\cdot\dot{\Xi}_b + \dot{\Xi}_a, \\
u_b &= \Xi_a e^{-2\Xi_b}\cdot\dot{\tilde\Xi}_c + \dot{\Xi}_b, \\
u_c &= e^{-2\Xi_b}\cdot\dot{\tilde\Xi}_c.
\end{split}
\end{align}
Considering this as a linear system of equations in variables $\Xi_a, \Xi_b, \tilde\Xi_c$, it is easy to conclude that
 the first equation in \eqref{eq:Controls} can be rewritten in the form
\begin{align}
\label{eq:Riccati}
\dot{\Xi}_a = u_a + 2u_b\cdot \Xi_a - u_c\cdot \Xi_a^2.
\end{align}
Since $\Xi_a(0) = 0$, we have proved the following theorem.

\begin{theoremprim}
\label{thm:RiccatiPrim}
For  fixed measurable functions $u_a, u_b, u_c:[0,T]\to\RR$, the function $\Xi_a:[0,T]\to\RR$, defined in Theorem \ref{thm:Main} by $\Xi_a(t) = \Upsilon^t(a\, \ExpShuffle{2\MainSeries})$, is (locally) the solution of the Riccati equation
$\dot\Xi_a(t) = u_a(t) + 2u_b(t)\cdot \Xi_a(t) - u_c(t)\cdot \Xi_a^2(t), \,\Xi_a(0) = 0.$
\end{theoremprim} 

Obviously, combining the second and the third equation in \eqref{eq:Controls}  we get
\begin{align}
\label{eq:TheTwo}
\dot{\Xi}_b &= u_b - u_c\cdot\Xi_a, &
\dot{\tilde\Xi}_c &= u_c\cdot e^{2{\Xi}_b}.
\end{align}
Integrating the first equation and using the definition \eqref{eq:Upsilon} of $\Upsilon^t$ we conclude
\begin{align*}
\Xi_b(t) &= \Upsilon^t(b - \WordsInSeriesOfLength{a}{}\, c). 
\end{align*}
Combining it with the defining formula $\Xi_b(t) = \Upsilon^t(\WordsInSeriesOfLength{b}{})$, we obtain $\WordsInSeriesOfLength{b}{} = b - \WordsInSeriesOfLength{a}{}\, c$ which is the second of the additional equality of Theorem \ref{thm:MainPrim}.
Finally, from the second equation in \eqref{eq:TheTwo}, we obtain that $\WordsInSeriesOfLength{c}{} = \ExpShuffle{2\MainSeries}\, c$. To see this, we use \eqref{eq:UpsilonIsAHomom} to get that
\begin{align*}
e^{2{\Xi}_b} = 1 + \sum_{k=1}^\infty \frac {2^k} {k!} \left(\Upsilon^t(\MainSeries)\right)^k = 
1 + \sum_{k=1}^\infty \frac {2^k} {k!} \Upsilon^t(\MainSeries^{\ShuffleProduct k}) =
\Upsilon^t(\ExpShuffle{2\MainSeries}).
\end{align*}
Therefore, integrating the second equality in \eqref{eq:TheTwo}, we conclude that
\begin{align*}
\tilde\Xi_c(t) = \int_0^t u_c(t_1)\cdot\Upsilon^{t_1}(\ExpShuffle{2\MainSeries})\, dt_1 = \Upsilon^{t_1}(\ExpShuffle{2\MainSeries}c),
\end{align*}
so $\WordsInSeriesOfLength{c}{} = \ExpShuffle{2\MainSeries}\, c$.

\begin{remark}
The system of three equations given by (\ref{eq:Riccati}) and (\ref{eq:TheTwo}) was actually considered by Redheffer \cite{Redheffer56Solutions} (see also \cite{Redheffer57Riccati}), who demonstrated a connection between this system and a  non-autonomous linear system on $\CC^2$  with $\LieAlgebra{sl}(2)$-matrix valued functions of time. Our results reproduce this phenomena and go further in producing a solution to the system,  i.e., Theorem \ref{thm:Riccati}.
\end{remark}

\subsection{Proof of Proposition \ref{prop:MainSeries} and Theorems \ref{thm:Main} and \ref{thm:Riccati}}

From Theorem \ref{thm:MainPrim} we easily conclude that 
\begin{align*}
\MainSeries = \WordsInSeriesOfLength{b}{} = b - \WordsInSeriesOfLength{a}{} c = b - a\,\ExpShuffle{2S}\, c,
\end{align*}
so $\MainSeries\in\Series{\RR}{\Letters}$ is a solution of the equation given in the proposition. We prove uniqueness of this solution. Let $\MainSeries = \sum_{n=1}^{\infty} \MainSeries_n$, where $\MainSeries_n\in\Series{\RR}{\Letters}$ is homogeneous of degree $n$, be a solution of the above equation. Then $\MainSeries_1 = b$, $\MainSeries_2 = -ac$, and for $n > 2$
\begin{align*}
\MainSeries_n = -a\left(2\MainSeries_{n-2} + \frac{2^2}{2!}\sum_{k+l = n-2}\MainSeries_k\ShuffleProduct\MainSeries_l + 
\frac{2^3}{3!}\sum_{k+l+m = n-2}\MainSeries_k\ShuffleProduct\MainSeries_l\ShuffleProduct\MainSeries_m +
\ldots\right)c,
\end{align*}
which gives us a recursive definition of the homogeneous summands in $\MainSeries$. Therefore $\MainSeries$ is defined uniquely.

Once again we emphasize that Theorem \ref{thm:MainPrim} together with Proposition \ref{prop:MainSeries} imply Theorem \ref{thm:Main}. It is also obvious that Theorem \ref{thm:RiccatiPrim} (proved in section \ref{sec:End}) is now equivalent to Theorem \ref{thm:Riccati}.

\subsection{Proof of Proposition \ref{prop:MainSeriesSymmetry}}

From Proposition \ref{prop:MainSeries} we know that the main series $\MainSeries\in\Series{\RR}{\Letters}$ is the unique solution of  
$\MainSeries = b - a\,\ExpShuffle{2\MainSeries}\, c.$ Since $\Antipode{(\cdot)}$ is an algebra antihomomorphism, and $\ShuffleProduct$ is commutative, we obtain $\Antipode \MainSeries = b - c\,\ExpShuffle{2\Antipode \MainSeries}\, a$. So
using the definition of $\FlipHom$,  we conclude that $\Antipode\FlipHom(\MainSeries) = b - a\,\ExpShuffle{2\Antipode\FlipHom(\MainSeries)}\, c$. It follows that $\Antipode\FlipHom(\MainSeries)$ satisfies the equation from Proposition \ref{prop:MainSeries}  for which $\MainSeries$ is the unique solution, so $\Antipode\FlipHom(S) = S$. Now the proposition follows easily from the equations established in Theorem \ref{thm:Main}.

\section{Concluding remarks}

In this article we presented the explicit solution of a general non-linear control-affine system of $\LieAlgebra{a}_1$-type. The solution was given in terms of the composition of three flows, with additional non-linear dependence on time, written with the use of certain well defined non-commutative power series $\MainSeries$ on three letters. 

Let us underline, that in order to formulate the results, we did not use the Hall basis structure, but only the set of Hall words. In other words, the solution of our problem does not depend on the edges of the trees. This fact makes the results easy to formulate, and it is also consistent with intuition, i.e., the solution should not depend on a chosen basis.

The solution given for the simplest simple algebra gives hope for similar solutions for other simple algebras. The first step is to find expressions for four rank-two simple Lie algebras of type $\LieAlgebra{a}_2, \LieAlgebra{b}_2, \LieAlgebra{c}_2$ and $\LieAlgebra{g}_2$. This will be a topic of the author's future research.

Finally, let us mention that the approach given in this article can not be easily generalized to the case of the special unitary algebra $\LieAlgebra{su}(2)$ (and $\LieAlgebra{su}(n)$), which would be helpful in studies of quantum (control) systems. The problem comes from the fact that generators of this Lie algebra (i.e., the Pauli matrices) are in a cyclic relation. It will definitely be interesting to overcome this problem.

\section*{Acknowledgements}

The author was partially supported by the Polish Ministry of Research and Higher Education grant N201 607540, 2011-2014.

\bibliographystyle{amsalpha}
\bibliography{bibliography}

\end{document}